\documentclass{article} 

\usepackage[T1]{fontenc}
\usepackage[utf8]{inputenc}

\usepackage{amsmath}
\usepackage{amssymb}
\usepackage{amsthm}
\usepackage{amsfonts}

\usepackage[margin=1in]{geometry}
\usepackage{tikz}
\usepackage{bbm}
\usepackage{setspace}
\usepackage{lmodern}
\usepackage{microtype}
\usepackage{titling}
\usepackage[title]{appendix}
\usepackage[hidelinks]{hyperref}

\newtheorem{theorem}{Theorem}
\newtheorem{proposition}{Proposition}
\newtheorem{corollary}{Corollary}

\theoremstyle{definition}

\newtheorem{remark}{Remark}

\newcommand*\diff{\mathop{}\!\mathrm{d}}
\newcommand{\stirlingone}[2]{\genfrac{[}{]}{0pt}{}{#1}{#2}}
\newcommand{\stirlingtwo}[2]{\genfrac{\lbrace}{\rbrace}{0pt}{}{#1}{#2}}

\newcommand{\idop}{\mathbbm{1}}
\newcommand{\RR}{\mathbb{R}}
\newcommand{\PP}{\mathbf{P}}
\newcommand{\EE}{\mathbf{E}}
\newcommand{\LL}{\mathcal{L}}
\newcommand{\ee}{\mathrm{e}}
\newcommand{\orig}{\mathbf{0}}
\newcommand{\X}{\mathbf{X}}

\thanksmarkseries{arabic}
\title{Occupation times on the legs of a diffusion spider}

\author{
Paavo Salminen%
\thanks{\AA bo Akademi University, \AA bo, Finland; 
\href{mailto:phsalmin@abo.fi}{phsalmin@abo.fi}}
\and 
David Stenlund%
\thanks{University of British Columbia, Vancouver~BC, Canada; 
\href{mailto:stenlund@math.ubc.ca}{stenlund@math.ubc.ca} \newline
The work by D.S. was in part funded through research grants from the Magnus Ehrnrooth Foundation and the Swedish Cultural Foundation in Finland. }
}

\date{}

\begin{document}

\maketitle
\setstretch{1.1}

\begin{abstract}
We study the joint moments of occupation times on the legs of a diffusion spider. 
Specifically, we give a recursive formula for the Laplace transform of the joint moments, which extends earlier results for a one-dimensional diffusion. 
For a Bessel spider, of which the Brownian spider is a special case, our approach yields an explicit formula for the joint moments of the occupation times. 


\end{abstract}

\medskip

\section{Introduction}

The process known as \emph{Walsh Brownian motion} was introduced by J.B.~Walsh in 1978 as an extension of the skew Brownian motion. 
The Walsh Brownian motion lives in~$\RR^2$, best expressed using polar coordinates. 
When away from the origin, the angular coordinate stays constant (so the process moves along a line), while the radial distance follows an excursion from 0 of a standard Brownian motion. 
Intuitively and roughly speaking, every time the process reaches the origin a new angle is randomly selected according to some distribution on~$[0,2\pi)$. 
This process was brilliantly described by Walsh in the following way~\cite{Walsh1978}:

\begin{quote}
It is a diffusion which, when away from the origin, is a Brownian motion along a ray, but which has what might be called a \emph{roundhouse singularity} at the origin: when the process enters it, it, like Stephen Leacock's hero, immediately rides off in all directions at once.
\end{quote}
\noindent
The construction of Walsh Brownian motion was described in more detail by Barlow, Pitman and Yor~\cite{BarlowPitmanYor1989walsh}. We refer also to Salisbury \cite{Salisbury1986} and Yano \cite{Yano2017}.

If the angle is selected according to a discrete distribution, then there are at most countably many rays on which the diffusion lives. 
The state space of the process then corresponds to a star graph with edges of infinite length, and we call such a graph a \emph{spider}. 
Thus the Walsh Brownian motion can, in this case, be seen as an early example of a diffusion on a graph, and we call this process a \emph{Brownian spider}. 
To make it more general, we can forego the requirement that the radial distance follows a Brownian motion and replace it with excursions from 0 of a regular reflected non-negative recurrent one-dimensional diffusion. 
In this paper, such a process is simply called a \emph{diffusion spider}, which can also be seen as an abbreviation for ``diffusion process on a spider''.  

Diffusions on graphs have been under intensive research at least since the pioneering work by Freidlin and Wentzell \cite{FreidlinWentzell1993}. 
We refer to Weber \cite{Weber2001} for earlier references, but also for a study in the direction of our paper. 
In addition to \cite{Walsh1978, BarlowPitmanYor1989walsh, BarlowPitmanYor1989arcsinus} concerning diffusions on spiders, we recall, in particular, the papers by 
Papanicolaou, Papageorgiou and Lepipas \cite{PapanicolaouPL2012}, 
Vakeroudis and Yor \cite{VakeroudisYor2012}, 
Fitzsimmons and Kuter \cite{FitzsimmonsKuter2014},  
Yano~\cite{Yano2017}, 
Csáki, Csörg\H{o}, Földes and Révész \cite{CsakiCFR2016, CsakiCFR2019}, 
Ernst \cite{Ernst2016}, 
Karatzas and Yan \cite{KaratzasYan2019}, 
Bayraktar and Zhang \cite{BayraktarZhang2021}, 
Atar and Cohen~\cite{AtarCohen2019}, 
Lempa, Mordecki and Salminen~\cite{LempaMordeckiSalminen2024}, 
and Bednarz, Ernst and Os\k{e}kowski~\cite{BednarzEO2024}. 

For results on occupation times and further earlier references, see \cite{Yano2017} where the joint law of the occupation times on legs of a diffusion spider (there called a ``multiray diffusion'') is characterized via a double Laplace transform formula generalizing the results in \cite{BarlowPitmanYor1989arcsinus} for a spider with excursions following a Bessel process. 
We also refer to~\cite{Yano2017} for a formula for the density of the joint law. 
In \cite{PapanicolaouPL2012} and  \cite{VakeroudisYor2012} are considered occupation times for a Brownian spider, and in \cite{CsakiCFR2019} focus is on limit theorems for local and occupation times for the Brownian  spider. 
A brief overview of the work by R\'ev\'esz et al.\ on Brownian spiders and random walks on spiders is given in~\cite{CsakiFoldes2024}.

In this paper, we study the joint moments of the occupation times on different legs of a diffusion spider. 
The theorems are extensions of our previous results for one-dimensional diffusions~\cite{SalminenStenlund2021}. 
Our method closely follows the earlier paper, namely, we apply a version of the Kac moment formula to obtain a recursive expression for the Laplace transform of the joint moments in the general case, and directly for the joint moments when the diffusion spider is self-similar. 
To derive  these expressions we need regularity properties of the Green function of a diffusion spider, which are obvious from the explicit form of the Green function derived in \cite{LempaMordeckiSalminen2024}. 
We solve the recursion for a Bessel spider---of which the Brownian spider is a special case---thereby obtaining an explicit formula for the joint moments of occupation times on the legs. 
We also briefly return to the original Walsh Brownian motion at the end of the paper.

\section{Preliminaries}
\label{sec_prelim}

\subsection{Linear diffusions}

To make the paper more self-contained, we first recall the basic facts from the theory of linear diffusions needed to introduce the concept of a diffusion spider. 
Let $X=(X_t)_{t\geq 0}$ be a linear diffusion living on $\RR_+=[0,+\infty)$. Let $\PP_x$ denote the probability measure associated with $X$ when initiated at $x\geq 0$.  For $y\geq 0$ introduce the first hitting time via 
\[
H_y: = \inf \{t\geq 0 : X_t=y\}.
\]
It is assumed that $X$ is regular and recurrent. Hence, for all $x\geq 0$ and $y\geq 0$ it holds that
\[
\PP_x\left( H_y<\infty\right)=1.
\]
Moreover, we suppose that 0 is a reflecting boundary and $+\infty$ is a natural boundary (for the boundary classification for linear diffusions, see \cite{ItoMcKean1974} and \cite{BorodinSalminen2015}).  
The $ \PP_x$-distribution of $H_y$ is characterized for $\lambda>0$ via the Laplace transform  

\begin{equation}
\label{LHIT}
\EE_x\left(\ee^{-\lambda H_y}\right) =
\begin{cases}
\dfrac{\varphi_\lambda (x)}{\varphi_\lambda (y)},& x\geq y,\\[10pt]
\dfrac{\psi_\lambda (x)}{\psi_\lambda (y)},& x\leq y,
\end{cases}
\end{equation}
 where $\EE_x$ refers to the expectation operator associated with $X$ and $\varphi_\lambda$ ($\psi_\lambda$) is a positive, continuous and decreasing (increasing) solution of the generalized differential equation 
\[
\frac d{dm}\frac d{dS}u=\lambda u.
\]
Here, $S$ and $m$ denote the scale function (strictly increasing and continuous) and the speed measure, respectively, associated with $X$. 
Under our assumptions, $m$ is a positive measure. 
To fix ideas, we also assume that $m$ does not have atoms and that $\varphi_\lambda$ and $\psi_\lambda$ are differentiable with respect to $S$. 
Recall that $\varphi_\lambda$ and $\psi_\lambda$ are unique up to multiplicative constants. 
We also introduce the diffusion $X^{\partial}$ with the same speed and scale as $X$ but for which 0 is a killing boundary. 
For $X^{\partial}$ there exist functions $\varphi^{\partial}_\lambda$ and $\psi^{\partial}_\lambda$ describing the distribution of $H_y$ for $X^{\partial}$ similarly as is done in \eqref{LHIT} for~$X$. 

Recall that 
\begin{equation}\label{recall}
\psi^{\partial}(0)=0, \quad \frac{d\psi}{d S}(0+)=0,\quad {\varphi^\partial\equiv\varphi},
\end{equation}
where the notation is shortened by omitting the subindex $\lambda$. Moreover, we normalize, as in~\cite{LempaMordeckiSalminen2024},
\begin{equation}
\label{norm}
S(0)=0,\quad \psi(0)=\varphi(0)=\varphi^\partial(0)=1, \quad \text{and}\quad \frac{d\psi^\partial}{d S}(0+)=1.
\end{equation}
As is well known, $X$ has a transition density $p$ with respect to $m$, i.e., for a Borel subset~$A$ of~$\RR_+$,
\begin{equation*}
\PP_x(X_t\in A)=\int_A p(t;x,y)\,m(dy), 
\end{equation*}
and the Green function (resolvent density) is given by
\begin{align}
\label{gr0}
\nonumber
g_\lambda(x,y)&:=\int_0^\infty \ee^{-\lambda t}\,p(t;x,y)\,dt\\
&=\begin{cases}
w_\lambda^{-1}\psi(y)\varphi(x),&0\leq y\leq x,\\
w_\lambda^{-1}\psi(x)\varphi(y),&0\leq x\leq y,
\end{cases}
\end{align}
with the Wronskian 
\[
w_\lambda=\frac{d\psi}{d S}(x)\varphi(x)-\frac{d\varphi}{d S}(x)\psi(x)=-\frac{d\varphi}{d S}(0+) .
\]

For later use, recall that a diffusion $X$ with starting point $X_0=0$ is called {self-similar} if for any $a> 0$ there exists $b> 0$ such that 
\[
(X_{at})_{t\geq 0} \overset{(d)}{=} (bX_t)_{t\geq 0}.
\]
The perhaps most well-known example of a self-similar diffusion is a standard Brownian motion starting in $0$, for which the above identity holds with $b=\sqrt{a}$.

\subsection{Diffusion spider}

Let $\Gamma\subset \RR^2$ be a star graph with one vertex at the origin of~$\RR^2$ and $R$ edges of infinite length meeting in the vertex (see Figure~\ref{fig_stargraph} for an example). 
Such a graph is here called a spider. 
The edges $L_1, \dotsc, L_R$ of the graph are known as the ``rays'' or---as hereafter called---the ``legs'' of the spider. 
The ordered pair $(x,i)$ describes the point on $\Gamma$ located on leg $L_i$ ($i=1,\dotsc,R$) at the distance $x\geq 0$ to the origin. 
We take the origin to be common to all legs, i.e., 
\[
(0,1)=(0,2)=\dots=(0,R), 
\]
so for simplicity we just write $\orig$ for the origin. 

\begin{figure}[htb]
\begin{center}
\begin{tikzpicture}
\node[circle,fill=black, scale=0.4] at (360:0mm) (center) {};
\foreach \n in {1,...,5}{
	\node at ({138+\n*(-360)/5}:20mm) (n\n) {};
	\node at ({138+\n*(-360)/5+10}:15mm) (lab\n) {\small $L_{\n}$};
	\draw [-stealth] (center)--(n\n);
}
\end{tikzpicture}
\caption{The graph of a diffusion spider with five legs.}
\label{fig_stargraph}
\end{center}
\end{figure}
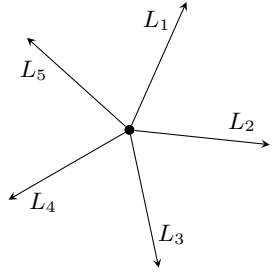

Let $X$ be the linear diffusion introduced above. 
On the graph $\Gamma$ we consider a stochastic process $\X:=(\X_t)_{t\geq 0}$ using the notation 
\[
\X_t:=(X_t,\rho_t), 
\]
where $\rho_t\in\{1,2,\dots,R\}$ indicates the leg on which $\X_t$ is located at time $t$ and $X_t$ is  the distance of $\X_t$ to the origin at time $t$ measured along the leg $L_{\rho_t}$. 
On each leg $L_i$, the process $\X$ behaves like the diffusion $X$ until it hits~$\orig$. 
The process $\X$ is called a (homogeneous) diffusion spider. We could allow different diffusions on the different legs (the inhomogeneous case), but do not do so in this paper. 
As part of the definition of the process, there are positive real numbers  $\beta_i$, $i=1,\dotsc,R$, such that $\sum_{i=1}^R\beta_i=1$. 
When $\X$ hits $\orig$, it  continues, roughly speaking, with probability $\beta_i$ onto leg $L_i$. 
We do not here discuss the rigorous construction of the process; for this, see the references given in the introduction. 
Notations $\PP_{(x,i)}$ and $\EE_{(x,i)}$ are used for the probability measure and expectation, respectively, when the diffusion spider starts in the point $(x,i)$, that is, on leg number $i$ and at a distance $x$ from the origin. 
As mentioned above, we write $\PP_\orig$ and $\EE_\orig$ without specifying a leg when the starting point is the origin. 

For the diffusion spider $\X$, we introduce its Green kernel (also called the resolvent kernel)  via 
\[
\mathbf{G}_\lambda u(x,i):=\int_0^{\infty}e^{-\lambda t}\EE_{(x,i)}\left(u(\X_t)\right)\,dt,
\]
where $(x,i)\in \Gamma$, $\lambda>0$ and $u\colon\Gamma\to\RR$ is a bounded measurable function. Moreover, define 
\begin{align}
&\mathbf{m}(dx,i):=\beta_im(dx),\quad i=1,\dots,R,\qquad \mathbf{m}(\{\orig\}):=0, \label{eq_speedscale}\\
&\mathbf{S}(dx,i):=\frac{1}{\beta_i}S(x). \nonumber
\end{align}
We call $\mathbf{m}$ and $\mathbf{S}$  the speed measure and the scale function, respectively, of $\X$. Clearly, on every leg $L_i$ of the diffusion spider,  
\[
\frac d{d\mathbf{m}}\frac d{d\mathbf{S}}=\frac d{dm}\frac d{dS}.
\]
Let $H_\orig = \inf \{t\geq 0 : \X_t=\orig\}$ be the first hitting time of $\orig$ for the diffusion spider $\X$. Since on every leg of the diffusion spider we have, loosely speaking, the same one-dimensional diffusion, for every $i=1,2,\dotsc, R$ and $x>0$, 
\begin{equation}
\label{CT}
 \EE_{(x,i)}(\ee^{-\lambda H_\orig}) =  \EE_{x}(\ee^{-\lambda H_0}) = \varphi_\lambda (x). 
\end{equation}
The following theorem, proved in \cite{LempaMordeckiSalminen2024}, states an explicit expression for the resolvent density of $\X$. 

\begin{theorem}
\label{Thrm:GreenKernel1}
The Green kernel of the diffusion spider $\X$ has a density~$\mathbf{g}_\lambda$ with respect to the speed measure~$\mathbf{m}$, which is 
given for $ x\geq 0$ and $y\geq 0$ by
\begin{equation}
\label{gk1}
\mathbf{g}_\lambda((x,i),(y,j))=
\begin{cases}
\varphi(y)\tilde{\psi}(x,i), &\quad x\leq y,\ i=j, \\[2pt]
\varphi(x)\tilde{\psi}(y,i), &\quad y\leq x,\ i=j, \\[2pt]
c_\lambda^{-1}\,\varphi(y)\varphi(x), &\quad i\neq j,
\end{cases}
\end{equation}
where 
\begin{equation}
\label{tildepsi}
\tilde{\psi}(x,i):=\frac1{\beta_i}\psi^{\partial}(x)+\frac1{c_\lambda}\varphi(x)
\end{equation}
and 
\begin{equation}
\label{cr1}
c_\lambda:=-\frac{d}{dS}\varphi(0+)>0.
\end{equation}
\end{theorem}

From the properties of the functions $\psi^{\partial}$ and $\varphi$, we immediately have the following result.

\begin{corollary}
\label{COR1}
The resolvent density (the Green function) $\mathbf{g}_\lambda$ given in \eqref{gk1} is continuous on $\Gamma$, and for every $i$ and~$j$,  
\begin{equation}
\label{CY}
\lim_{(x,i)\to\orig}\mathbf{g}_\lambda((x,i),(y,j))=\mathbf{g}_\lambda(\orig,(y,j))
= \frac1{c_\lambda}\varphi(y)
= g_\lambda(0,y). 
\end{equation}
\end{corollary}

\section{Kac's moment formula}

The tool that we will use to obtain the recursive expression for the joint moments is an extended variant of the Kac moment formula. 
Let $Y$ be a regular diffusion taking values on an interval $E$. 
In spite of some conflict with our earlier notation, here we also let $m$, $p$, and  $\EE_x$ ($x\in E$) denote the speed measure, transition density and expectation operator, respectively, associated with~$Y$. 
Moreover, let $V:E\mapsto\RR$ be a measurable and bounded function and define for $t>0$ the additive functional
\[
A_t(V) := \int_0^t V(Y_s) \diff s. 
\]
The moment formula by M. Kac for integral functionals, see \cite{Kac1951}, i.e.,
\begin{equation*}
\EE_x\bigl((A_t(V))^n\bigr) = n \int_E m(\diff y) \int_0^t p(s;x,y) V(y) \EE_y\bigl((A_{t-s}(V))^{n-1}\bigr) \diff s,
\end{equation*}
can easily be extended into the following formula for the expected value of a product of powers of different functionals. 

\begin{proposition} \label{prop_Kac}
Let $V_1, \dotsc, V_N$ be measurable and bounded functions on $E$. For $t>0$, $x\in E$ and $n_1, \dotsc, n_N\in\{1,2,\dotsc\}$,
\begin{equation}\label{eq_Kac_multi}
\EE_x\left(\prod_{k=1}^N (A_t(V_k))^{n_k}\right) = \sum_{k=1}^N n_k \int_E m(\diff y) \int_0^t p(s;x,y) V_k(y) \EE_y\left(\frac{\prod_{i=1}^N (A_{t-s}(V_i))^{n_i}}{A_{t-s}(V_k)}\right) \diff s.
\end{equation}
\end{proposition}

\begin{proof}
We prove the statement for $N=2$, as the proof is analogous for any $N>2$. 
Restating the equation above with $N=2$, what we want to prove is that
\begin{align*}
\EE_x\left(A_t(V_1))^{n_1}A_t(V_2))^{n_2}\right) = &n_1 \int_E m(\diff y) \int_0^t p(s;x,y) V_1(y) \EE_y\left(A_{t-s}(V_1))^{n_1-1}A_{t-s}(V_2))^{n_2}\right) \diff s \\
&+ n_2 \int_E m(\diff y) \int_0^t p(u;x,y) V_2(y) \EE_y\left(A_{t-u}(V_1))^{n_1}A_{t-u}(V_2))^{n_2-1}\right) \diff u.
\end{align*}
Expanding the powers of $A_t(V_1)$ and $A_t(V_2)$ on the left hand side, we get 
\begin{align*}
&\EE_x\left(A_t(V_1))^{n_1}A_t(V_2))^{n_2}\right) \\
& = \EE_x\left( \int_0^t \diff s_1 \dotsm \int_0^t \diff s_{n_1} \int_0^t \diff u_1 \dotsm \int_0^t \diff u_{n_2} \,V_1(X_{s_1})\dotsm V_1(X_{s_{n_1}}) \cdot V_2(X_{u_1})\dotsm V_2(X_{u_{n_2}}) \right) \\[6pt]
& = n_1 \EE_x\left( \int_0^t \diff s_1 V_1(X_{s_1}) \int_{s_1}^t \diff s_2 \dotsm \int_{s_1}^t \diff u_{n_2} \,V_1(X_{s_2})\dotsm V_2(X_{u_{n_2}}) \right) \\[-1pt]
& \qquad + {n_2} \EE_x\left( \int_0^t \diff u_{n_2} V_2(X_{u_{n_2}}) \int_{u_{n_2}}^t \diff s_1 \dotsm \int_{u_{n_2}}^t \diff u_{{n_2}-1} \,V_1(X_{s_1})\dotsm V_2(X_{u_{{n_2}-1}}) \right) \\[6pt]
& = n_1 \int_0^t \diff s_1 \EE_x\left( V_1(X_{s_1}) \int_{0}^{t-s_1} \diff s_2 \dotsm \int_{0}^{t-s_1} \diff u_{n_2} \,V_1(X_{s_1+s_2})\dotsm V_2(X_{s_1+u_{n_2}}) \right) \\[-1pt]
& \qquad + {n_2} \int_0^t \diff u_{n_2} \EE_x\left( V_2(X_{u_{n_2}}) \int_{0}^{t-u_{n_2}} \diff s_1 \dotsm \int_{0}^{t-u_{n_2}} \diff u_{{n_2}-1} \,V_1(X_{u_{n_2}+s_1})\dotsm V_2(X_{u_{n_2}+u_{{n_2}-1}}) \right) \\[6pt]
& = n_1 \int_0^t \diff s_1 \int_E m(\diff y) p(s_1;x,y) V_1(y) \EE_y \left( A_{t-s_1}(V_1))^{n_1-1}A_{t-s_1}(V_2))^{{n_2}} \right) \\[-1pt]
&\qquad + {n_2} \int_0^t \diff u_{n_2} \int_E m(\diff y) p(u_{n_2};x,y) V_2(y) \EE_y \left( A_{t-u_{n_2}}(V_1))^{n_1}A_{t-u_{n_2}}(V_2))^{{n_2}-1} \right),
\end{align*}
where we have used the symmetry of the integrand in the second step and the strong Markov property in the final step. 
After a change of order of integration, the desired result follows. 
\end{proof}

\begin{remark}
\label{rem_kaczeros}
In the proof above, it is assumed that all values $n_1, \dotsc, n_N$ are strictly positive integers. 
However, \eqref{eq_Kac_multi} may hold  even if some (but not all) of these values are zero. 
Note that the factor $n_k$ in the terms on the right hand side ensures that any term with $n_k=0$ will not contribute, as long as the integral in that term is convergent. 
Suppose that for any starting point $X_0=y$ of the underlying diffusion for which $V_k(y) \neq 0$, it almost surely holds that $A_t(V_k) > 0$ when $t>0$, and $|A_t(V_k)| \geq |A_t(V_i)|, \forall i=1,\dotsc, N$ when $t\to 0$. 
Then the denominator $A_{t-s}(V_k)$ inside the expected value is nonzero except possibly when $s\to t$, in which case all factors in the numerator (and there is at least one) tend to zero as well, and at least as fast. 
The first condition holds, for instance, if $V_k$ is non-negative everywhere and continuous in~$y$, as the diffusion~$X$ is assumed to be regular.
Observe that one choice of functions $V_k$ that satisfies both these conditions is to take the indicator functions on the legs of a spider, i.e., $V_k = \idop_{L_k}$, which is the case of interest in this paper. 
Technically, the point $\orig$ should be excluded from the indicator functions for the conditions to hold also when $\orig$ is the starting point, but assuming that $\orig$ is not a sticky point (as we here do), that will not make a difference.
\end{remark}

\section{Recursive formula for moments}

Let $(\X_t)_{t\geq 0}$ be a diffusion spider with $R \geq 2$ legs meeting in the point~$\orig$, and let 
\[
A_t^{(i)} := \int_0^t \idop_{L_i}(X_s) \diff s
\]
be the occupation time on leg number $i$ up to time $t$. 
Note that if $X$ is self-similar, it follows for any $i$ and any fixed $t\geq 0$ that 
\[
A_t^{(i)} \overset{\mathrm{(d)}}{=} t A_1^{(i)}, 
\] 
meaning that for any such diffusion spider, we can equally well consider the occupation time up to time~1 instead of a general (fixed) time~$t$. 

In this section are presented formulas for recursively finding the moments of occupation times on the legs of a diffusion spider. 
In the first and shorter subsection we recapitulate the result for moments of a single occupation time, which is presented in our earlier paper~\cite{SalminenStenlund2021}. 
In the second subsection this result is extended to joint moments of multiple occupation times.

\subsection{Occupation time on a single leg}
\label{sec_singleleg}

As pointed out in Section~6.4 of~\cite{SalminenStenlund2021}, the occupation time on a single leg~$L_i$ of a (homogeneous) diffusion spider has the same law as the occupation time on the positive half-line of a one-dimensional skew diffusion process, with the skewness parameter given by~$\beta_i$. 
Namely, the spider is mapped onto $\RR$ so that the leg~$L_i$ corresponds to the positive half-line $[0,\infty)$, while all other legs are grouped together into a single second leg with parameter $\sum_{k\neq i}\beta_k = 1-\beta_i$, which then is taken to be the negative half-line $(-\infty, 0]$. 
When considering the occupation time on the leg~$L_i$, we can therefore equally well consider the occupation time on $[0,\infty)$ of a one-dimensional diffusion. 

The following results are shown in the previous paper~\cite{SalminenStenlund2021}, although here slightly modified to comply with the notation for diffusion spiders introduced in Section~\ref{sec_prelim}. 
In particular, note that the one-dimensional diffusion $X$ on $\RR$ with speed measure $m$, in the setting of the other paper, here really corresponds to a two-legged diffusion spider $\X$ with speed measure $\mathbf{m}$, which is why $m(dx)$ has been replaced by $\beta_i m(dx)$ in~\eqref{eq_Dk} below, in accordance with~\eqref{eq_speedscale}. 

\begin{theorem}\label{thm_singlerec}
For $n \geq1$, 
\begin{align} 
\label{eq_Laplacerec1dim}
\nonumber
\LL_t\Bigg\{ \EE_\orig\left((A_t^{(i)})^{n}\right) \Bigg\}(\lambda) = \sum_{k=1}^{n}& \frac{n! D_k^{(i)}(\lambda)}{(n-k)! \lambda^k} \LL_t\Bigg\{ \EE_\orig\left( (A_t^{(i)})^{n-k} \right) \Bigg\}(\lambda)\\ 
&+ \frac{n!}{\lambda^{n-1}} \LL_t\Bigg\{ \EE_\orig(A_t^{(i)}) \Bigg\}(\lambda)
- \frac{n!}{\lambda^{n+1}} \sum_{k=1}^{n} D_k^{(j)}(\lambda),
\end{align}
where $\LL_t$ denotes the Laplace transform with respect to $t$ of the function in curly brackets, $\lambda$ is the Laplace parameter, and  
\begin{equation}\label{eq_Dk}
D_k^{(i)}(\lambda) := \frac{\lambda^k}{(k-1)!} \int_0^\infty \mathbf{g}_\lambda(\orig,(y,i)) \EE_{(y,i)}(H_\orig^{k-1} \ee^{-\lambda H_\orig}) \beta_i\,m(\diff y). 
\end{equation}
Furthermore, if $X$ is self-similar, then  for any $\lambda>0$, 
\begin{equation} \label{eq_mom1dim}
\EE_\orig\left((A_1^{(i)})^{n}\right) = \EE_\orig\left(A_1^{(i)}\right) - \sum_{k=1}^{n} D_k^{(i)}(\lambda) \left( 1 - \EE_\orig\left( (A_1^{(i)})^{n-k} \right) \right), 
\end{equation}
and, in particular, $D_k^{(i)}(\lambda)$ does not depend on $\lambda$ for any $i=1,\dotsc,R$ and $k=1,2,\dotsc$.  
\end{theorem}

The proof follows the lines of the proof of Theorem~2 in \cite{SalminenStenlund2021} with some minor notational changes, as indicated above. 

\begin{remark}
Note the following: 
\begin{enumerate}
\item There is a sign change in the factor $D_k^{(i)}(\lambda)$ as defined in~\eqref{eq_Dk} compared to~\cite{SalminenStenlund2021}. 
This is a deliberate change that---in hindsight---could have been used in the cited paper as well.
\item The variable $\lambda$ is not at all present on the left hand side of \eqref{eq_mom1dim}, and (using induction) we can conclude that the factors $D_k^{(i)}(\lambda)$ cannot depend on $\lambda$ either. 
Thus, the value $\lambda >0$ can be chosen arbitrarily. 
As a side note, this is the reason why a factor $\lambda^k$ is included in the expression for $D_k^{(i)}(\lambda)$ in~\eqref{eq_Dk}. 
\end{enumerate}
\end{remark}

\subsection{Joint moments}
\label{sec_joint}

The result in the previous section (from~\cite{SalminenStenlund2021}) is here extended to a recursive formula for the Laplace transforms of the joint moments of the occupation times on multiple legs of a diffusion spider. 
For self-similar spiders we have a recursive formula directly for the joint moments, as in the case of the occupation time on one leg, cf.~\eqref{eq_mom1dim} in Theorem~\ref{thm_singlerec}. 

\begin{theorem}\label{thm_momrec}
For $r\in\{2,\dotsc,R\}$ and $n_1,\dotsc,n_r \geq1$, 
\begin{equation} \label{eq_Laplacerec}
\LL_t\Bigg\{ \EE_\orig\left(\prod_{i=1}^r (A_t^{(i)})^{n_i}\right) \Bigg\}(\lambda) = \sum_{i=1}^r \sum_{k=1}^{n_i} \frac{{n_i}! D_k^{(i)}(\lambda)}{(n_i-k)! \lambda^k} \LL_t\Bigg\{ \EE_\orig\left( \frac{\prod_{j=1}^r (A_t^{(j)})^{n_j} }{(A_t^{(i)})^{k}} \right) \Bigg\}(\lambda),
\end{equation}
where $D_k^{(i)}$ is defined in~\eqref{eq_Dk}.   
If $X$ is self-similar, then  for any $\lambda>0$, 
\begin{equation} \label{eq_selfrec}
\EE_\orig\left(\prod_{i=1}^r (A_1^{(i)})^{n_i}\right) = \sum_{i=1}^r \sum_{k=1}^{n_i} \dfrac{\binom{n_i}{k}}{\binom{n_1+\dotsc+n_r}{k}} D_k^{(i)}(\lambda) \EE_\orig\left( \frac{\prod_{j=1}^r (A_1^{(j)})^{n_j} }{(A_1^{(i)})^{k}} \right). 
\end{equation}
\end{theorem}

\begin{proof} 
We first remark that the expression on the right hand side of \eqref{eq_Laplacerec} is well defined since $\PP_\orig\left(A_t^{(i)}>0\right)=1$ for all $t>0$ and $i=1,2,\dotsc,R$. 
The procedure below closely follows the proof of Theorem~2 in~\cite{SalminenStenlund2021}, with the main difference being that instead of the original Kac's moment formula we use the generalized version given in Proposition~\ref{prop_Kac}. 
Proving the result for $r=2$ should be sufficient, since the method is the same for any higher values of~$r$ (you only need to include more terms of similar form). 

With $r=2$, the claimed identity \eqref{eq_Laplacerec} can be written
\begin{align} \label{eq_Laplacerec_N2}
\LL_t\biggl\{ \EE_\orig\left((A_t^{(1)})^{{n_1}} (A_t^{(2)})^{{n_2}} \right) \biggr\}(\lambda) &= \sum_{k=1}^{{n_1}} \frac{{{n_1}}! D_k^{(1)}(\lambda)}{({n_1}-k)! \lambda^k} \LL_t\biggl\{ \EE_\orig\left( (A_t^{(1)})^{{n_1}-k} (A_t^{(2)})^{{n_2}} \right) \biggr\}(\lambda) \nonumber \\
&\quad + \sum_{k=1}^{{n_2}} \frac{{{n_2}}! D_k^{(2)}(\lambda)}{({n_2}-k)! \lambda^k} \LL_t\biggl\{ \EE_\orig\left( (A_t^{(1)})^{{n_1}} (A_t^{(2)})^{{n_2}-k} \right) \biggr\}(\lambda). 
\end{align}
Note that the numbering of the diffusion spider legs is arbitrary, so any two legs could be considered, even though they are here numbered 1 and 2. 

Assume first that the diffusion starts on one of the legs at a distance $x$ from the origin. 
Without loss of generality, we will write that it starts on the first leg $L_1$. 
Before the diffusion hits~$\orig$ for the first time, the occupation time on the starting leg $L_1$ equals the entire elapsed time, while the occupation time on any other leg stays zero. 
Thus, applying the strong Markov property at the first hitting time of~$\orig$, we have
\[
A^{(1)}_t = 
\begin{cases}
H_\orig + A^{(1)}_{t-H_\orig} \circ \theta_{H_\orig}, & H_\orig < t, \\
t, & H_\orig \geq t, 
\end{cases}
\]
while for any other leg, 
\[
A^{(k)}_t = 
\begin{cases}
A^{(k)}_{t-H_\orig} \circ \theta_{H_\orig}, & H_\orig < t, \\
0, & H_\orig \geq t, 
\end{cases}
\quad (k>1). 
\]
where $\theta_t$ is the usual shift operator. 
From this we obtain, for any ${n_1},{n_2}\geq 1$, 
\begin{equation}\label{eq_A1n1A2n2}
(A_t^{(1)})^{{n_1}} (A_t^{(2)})^{{n_2}} = 
\begin{cases}
\sum_{k=0}^{n_1} \binom{{n_1}}{k} H_\orig^k (A^{(1)}_{t-H_\orig}\circ \theta_{H_\orig})^{{n_1}-k} (A^{(2)}_{t-H_\orig}\circ \theta_{H_\orig})^{{n_2}}, & H_\orig<t, \\
0, & H_\orig \geq t, 
\end{cases} 
\end{equation}
and
\[
\EE_{(x,1)} \left( (A_t^{(1)})^{{n_1}} (A_t^{(2)})^{{n_2}} \right) = \sum_{k=0}^{n_1} \binom{{n_1}}{k} \int_0^t \EE_\orig \left( (A^{(1)}_{t-s})^{{n_1}-k} (A^{(2)}_{t-s})^{{n_2}} \right) s^k f(x,1;s) \diff s, 
\]
where $f(x,1;t)$ denotes the $\PP_{(x,1)}$-density of $H_\orig$. 
Taking the Laplace transform with respect to $t$, we first recall that
\[
\LL_t \{ f(x,1;t) \}(\lambda) = \int_0^\infty \ee^{-\lambda t} \PP_x(H_\orig \in \diff t) = \EE_{(x,1)}\bigl(\ee^{-\lambda H_\orig}\bigr), 
\]
and 
\[
\LL_t \{ t^k f(x,1;t) \}(\lambda) = (-1)^k \frac{\diff^k}{\diff \lambda^k} \LL_t \{ f(x,1;t) \}(\lambda) = \EE_{(x,1)}\bigl(H_\orig^k \ee^{-\lambda H_\orig}\bigr), 
\]
so that we get, using the formula for the Laplace transform of a convolution, 
\begin{equation}\label{eq_Laplx-0}
\LL_t \biggl\{ \EE_{(x,1)} \left( (A_t^{(1)})^{{n_1}} (A_t^{(2)})^{{n_2}} \right) \biggr\}(\lambda) = \sum_{k=0}^{n_1} \binom{{n_1}}{k} \EE_{(x,1)}\bigl(H_\orig^k \ee^{-\lambda H_\orig}\bigr) \LL_t \biggl\{ \EE_\orig \left( (A^{(1)}_{t})^{{n_1}-k} (A^{(2)}_{t})^{{n_2}}\right)  \biggr\}(\lambda). 
\end{equation}
On the right hand side of the equation, the starting point $(x,1)$ is only part of the expression involving the hitting time $H_\orig$, while the Laplace transform of the joint moments instead has the starting point~$\orig$. 
The equation above is now used together with Kac's moment formula to obtain a recursive formula for (the Laplace transform of) the joint moments of the occupation times on the different legs. 
Proposition~\ref{prop_Kac} yields
\begin{align*}
\LL_t \biggl\{ \EE_{(x,1)} \left( (A_t^{(1)})^{{n_1}} (A_t^{(2)})^{{n_2}} \right) \biggr\}(\lambda) &= {n_1} \int_0^\infty \mathbf{g}_\lambda((x,1),(y,1)) \LL_t \biggl\{ \EE_{(y,1)} \left( (A_t^{(1)})^{{n_1}-1} (A_t^{(2)})^{{n_2}} \right) \biggr\}(\lambda) \, \beta_1\,m(\diff y)  \\
&+ {n_2} \int_0^\infty \mathbf{g}_\lambda((x,1),(y,2)) \LL_t \biggl\{ \EE_{(y,2)} \left( (A_t^{(1)})^{{n_1}} (A_t^{(2)})^{{n_2}-1} \right) \biggr\}(\lambda) \,\beta_2\, m(\diff y), 
\end{align*}
and inserting the expression in~\eqref{eq_Laplx-0} on both sides gives
\begin{align*}
&\sum_{k=0}^{{n_1}} \binom{{n_1}}{k} \EE_{(x,1)}\bigl(H_\orig^k \ee^{-\lambda H_\orig}\bigr) \LL_t \biggl\{ \EE_\orig \left( (A^{(1)}_{t})^{{n_1}-k} (A^{(2)}_{t})^{{n_2}} \right) \biggr\}(\lambda) \\
&= {n_1} \sum_{k=0}^{{{n_1}}-1} \binom{{n_1}-1}{k} \LL_t \biggl\{ \EE_\orig \left( (A^{(1)}_{t})^{{n_1}-k-1} (A^{(2)}_{t})^{{n_2}} \right) \biggr\}(\lambda) \int_0^\infty \mathbf{g}_\lambda((x,1),(y,1)) \EE_{(y,1)}\bigl(H_\orig^k \ee^{-\lambda H_\orig}\bigr) \, \beta_1\,m(\diff y)  \\
&+ {n_2} \sum_{k=0}^{{n_2}-1} \binom{{n_2}-1}{k} \LL_t \biggl\{ \EE_\orig \left( (A^{(1)}_{t})^{{n_1}} (A^{(2)}_{t})^{{n_2}-k-1} \right) \biggr\}(\lambda) \int_0^\infty \mathbf{g}_\lambda((x,1),(y,2)) \EE_{(y,2)}\bigl(H_\orig^k \ee^{-\lambda H_\orig}\bigr) \,\beta_2\, m(\diff y). 
\end{align*}
Note that the second integral is taken over points $(y,2)$ that lie on the second leg~$L_2$, which is why in that case the expression in~\eqref{eq_Laplx-0} is inserted with the roles of $n_1$ and $n_2$ interchanged. 
We now let $x\to 0$ on both sides of the equation above. For this, recall from \cite{SalminenStenlund2021} (Lemma~1) that 
\[
\lim_{x\to 0} \EE_{(x,1)}\bigl(H_\orig^k \ee^{-\lambda H_\orig}\bigr) = 
\begin{cases}
1, & k=0, \\
0, & k\geq 1,
\end{cases} 
\]
and by Corollary \ref{COR1} we may take the limit inside the integrals to obtain
\begin{align*}
&\LL_t \biggl\{ \EE_\orig \left( (A^{(1)}_{t})^{{n_1}} (A^{(2)}_{t})^{{n_2}} \right) \biggr\}(\lambda) \\
&= \sum_{k=1}^{{n_1}} \frac{{n_1}!}{({n_1}-k)!} \LL_t \biggl\{ \EE_\orig \left( (A^{(1)}_{t})^{{n_1}-k} (A^{(2)}_{t})^{{n_2}} \right) \biggr\}(\lambda) \frac{1}{(k-1)!} \int_0^\infty \mathbf{g}_\lambda(\orig,(y,1)) \EE_{(y,1)}\bigl(H_\orig^{k-1} \ee^{-\lambda H_\orig}\bigr) \,\beta_1\, m(\diff y) \\
&+ \sum_{k=1}^{{n_2}} \frac{{n_2}!}{({n_2}-k)!} \LL_t \biggl\{ \EE_\orig \left( (A^{(1)}_{t})^{{n_1}} (A^{(2)}_{t})^{{n_2}-k} \right) \biggr\}(\lambda) \frac{1}{(k-1)!} \int_0^\infty \mathbf{g}_\lambda(\orig,(y,2)) \EE_{(y,2)}\bigl(H_\orig^{k-1} \ee^{-\lambda H_\orig}\bigr) \,\beta_2\, m(\diff y), 
\end{align*}
where also the summation index is changed. 
This is equivalent to~\eqref{eq_Laplacerec_N2} when introducing $D_k^{(i)}(\lambda)$ as defined in \eqref{eq_Dk}, and this proves the first part of the theorem. 

For a self-similar spider,  
\[
(A^{(1)}_{t})^{{n_1}} (A^{(2)}_{t})^{{n_2}} \overset{\mathrm{(d)}}{=} t^{{n_1}+{n_2}} (A^{(1)}_{1})^{{n_1}} (A^{(2)}_{1})^{{n_2}}, 
\]
and, in this case,
\[
\LL_t \biggl\{ \EE_\orig \left( (A^{(1)}_{t})^{{n_1}} (A^{(2)}_{t})^{{n_2}} \right) \biggr\}(\lambda) = \frac{({n_1}+{n_2})!}{\lambda^{{n_1}+{n_2}+1}} \EE_\orig \left( (A^{(1)}_{1})^{{n_1}} (A^{(2)}_{1})^{{n_2}} \right).
\]
Hence, for self-similar spiders, the expression in~\eqref{eq_Laplacerec_N2} simplifies to 
\begin{align}
\label{LA}
\nonumber
\EE_\orig \left( (A^{(1)}_{1})^{{n_1}} (A^{(2)}_{1})^{{n_2}} \right) &= \sum_{k=1}^{{n_1}} \frac{\binom{{n_1}}{k}}{\binom{{n_1}+{n_2}}{k}} D_k^{(1)}(\lambda) \EE_\orig \left( (A^{(1)}_{1})^{{n_1}-k} (A^{(2)}_{1})^{{n_2}} \right)\\
& + \sum_{k=1}^{{n_2}} \frac{\binom{{n_2}}{k}}{\binom{{n_1}+{n_2}}{k}} D_k^{(2)}(\lambda) \EE_\orig \left( (A^{(1)}_{1})^{{n_1}} (A^{(2)}_{1})^{{n_2}-k} \right). 
\end{align}
This proves the second part of the theorem for $r=2$, and the proof is easily extended for higher~$r$. 
\end{proof}

The result in Theorem~\ref{thm_momrec} tells us that if we know the Green kernel of the diffusion spider $\X$, we can recursively compute any joint moments of the occupation times on a number of legs. 
Recall also from~\eqref{CT} that for $y>0$
\[
\EE_{(y,i)}\bigl(H_\orig^{k} \ee^{-\lambda H_\orig}\bigr) = 
\EE_{y}\bigl(H_0^{k} \ee^{-\lambda H_0}\bigr)
= (-1)^{k} \frac{\diff^k}{\diff \lambda^k} \frac{\varphi_{\lambda}(y)}{\varphi_{\lambda}(0)},
\]
i.e., we have all ingredients needed for computing the factors $D^{(i)}_k$ using the integral expression in~\eqref{eq_Dk}. 

\begin{remark}
\label{rem_r1}
Despite the similarity, Theorem~\ref{thm_singlerec} does not follow by letting $r=1$ in Theorem~\ref{thm_momrec}. 
On the right hand sides of equations \eqref{eq_Laplacerec1dim} and~\eqref{eq_mom1dim} are not only the parts corresponding to \eqref{eq_Laplacerec} and~\eqref{eq_selfrec}, respectively, with $r=1$, but also some additional terms. 
This difference seems to stem from the fact that up to the first hitting time of~$\orig$, the occupation time on the starting leg is equal to the time elapsed (and, hence, positive), while the occupation time on any other leg is zero. 
Any product of occupation times on more than one leg is, therefore, also zero up to time~$H_\orig$, as can be seen in~\eqref{eq_A1n1A2n2}, and even though we let the starting point tend to~$\orig$ when deriving the aforementioned theorems, there remains still a component which is nonzero in the single leg case but zero for the joint moments. 
With this in mind, Theorem~\ref{thm_momrec} should be seen not as a replacement of Theorem~\ref{thm_singlerec} but as a complement to it. 
\end{remark}

\section{Moment generating function}

The generalized version of Kac's moment formula in Proposition~\ref{prop_Kac} can also be used to derive a moment generating function of the occupation times on the legs of a diffusion spider up to an exponential time~$T$.  

\begin{theorem}\label{thm_momgenspider}
Let $T$ be exponentially distributed with mean $1/\lambda$, $\lambda>0$, and independent of $\X$. 
Then, for any $z_1,\dotsc,z_R \geq 0$, 

\begin{equation} 
\label{eq_momgenspider}
\EE_\orig\Biggl(\mathrm{exp}\biggl(- \sum_{i=1}^R z_i A_T^{(i)}\biggr)\Biggr) 
= \frac{\displaystyle 1-\sum_{j=1}^R \frac{\lambda z_j}{\lambda + z_j} \int_0^\infty \mathbf{g}_\lambda(\orig,(y,j)) \left(1 - \EE_{(y,j)}(\ee^{-(\lambda+z_j) H_\orig}) \right)\beta_j\, m(\diff y)}{\displaystyle 1+ \sum_{j=1}^R z_j \int_0^\infty \mathbf{g}_\lambda(\orig,(y,j)) \EE_{(y,j)}(\ee^{-(\lambda+z_j) H_\orig})\beta_j\, m(\diff y)}. 
\end{equation}
\end{theorem}

\begin{proof}
If the diffusion spider starts in the point $(x,j)$, that is, on a particular leg $L_j$, then up to the time $H_\orig$ the occupation time on the leg $L_j$ is equal to the time elapsed, while the occupation time on any other leg is zero. Hence, for any $i\in\{1,\dotsc,R\}$, 
\[
A^{(i)}_t = 
\begin{cases}
H_\orig \, \idop_{\{i=j\}} + A^{(i)}_{t-H_\orig} \circ \theta_{H_\orig}, & H_\orig < t, \\
t \, \idop_{\{i=j\}}, & H_\orig \geq t. 
\end{cases}
\]
To prove the theorem, we first consider the left hand side of~\eqref{eq_momgenspider} with a general starting point~$(x,j)$ instead of~$\orig$ and split the analysis into the two events whether the diffusion spider hits~$\orig$ before the exponential time~$T$ or not. 
This gives the two parts
\begin{align*}
\EE_{(x,j)}\left(\mathrm{exp}\bigg(- \sum_{i=1}^R z_i A_T^{(i)}\bigg) ; H_\orig \geq T \right) 
&= \EE_{(x,j)}\left(\ee^{-z_j T} ; H_\orig \geq T \right) \\
&= \frac{\lambda}{\lambda + z_j} \left( 1 - \EE_{(x,j)}\bigg(\ee^{-(\lambda + z_j) H_\orig} \bigg) \right)
\end{align*}
and
\begin{align*}
\EE_{(x,j)}\left(\mathrm{exp}\bigg(- \sum_{i=1}^R z_i A_T^{(i)}\bigg) ; H_\orig < T \right) 
&= \EE_{(x,j)}\left(\ee^{-z_j H_\orig} ; H_\orig < T \right) \EE_\orig\left(\mathrm{exp}\bigg(- \sum_{i=1}^R z_i A_T^{(i)}\bigg)\right) \\
&= \EE_{(x,j)}\bigg(\ee^{-(\lambda + z_j) H_\orig} \bigg) \EE_\orig\left(\mathrm{exp}\bigg(- \sum_{i=1}^R z_i A_T^{(i)}\bigg)\right),
\end{align*}
where in the second part we have used the strong Markov property to restart the process when it first hits $\orig$, as well as the memoryless property of the exponential distribution. 
Combining both parts gives
\begin{equation}
\label{eq_momgendepx}
\EE_{(x,j)}\left(\mathrm{exp}\bigg(- \sum_{i=1}^R z_i A_T^{(i)}\bigg) \right) = \frac{\lambda}{\lambda + z_j}
+ \left(\EE_\orig\left(\mathrm{exp}\bigg(- \sum_{i=1}^R z_i A_T^{(i)}\bigg)\right) - \frac{\lambda}{\lambda + z_j}\right) \EE_{(x,j)}\bigg(\ee^{-(\lambda + z_j) H_\orig} \bigg). 
\end{equation}
The significance of this expression is that on the right hand side the dependency on the starting position $(x,j)$ is contained only in a function of the first hitting time $H_\orig$, while the moment generating function of the occupation times has the starting point~$\orig$ instead. 

As the left hand side of~\eqref{eq_momgendepx} can be written
\begin{equation}\label{eq_mgfTasLaplace}
\EE_{(x,j)}\left(\mathrm{exp}\bigg(- \sum_{i=1}^R z_i A_T^{(i)}\bigg) \right) 
= \int_0^\infty \EE_{(x,j)}\left(\mathrm{exp}\bigg(- \sum_{i=1}^R z_i A_t^{(i)}\bigg) \right) \lambda \ee^{-\lambda t} \diff t,
\end{equation}
we will, for the moment, consider the expression with a fixed time~$t$ rather than the exponential time~$T$. Expanding as a sum of joint moments, we get 
\begin{align}
\EE_{(x,j)}\left(\mathrm{exp}\bigg(- \sum_{i=1}^R z_i A_t^{(i)}\bigg)\right) 
&= \sum_{N=0}^\infty \frac{(-1)^N}{N!} \EE_{(x,j)}\left( \left(\sum_{i=1}^R z_i A_t^{(i)}\right)^N \right) \nonumber\\
&= \sum_{N=0}^\infty \frac{(-1)^N}{N!}\EE_{(x,j)}\left( \sum_{\substack{n_1, \dotsc, n_R \geq 0 \\[2pt] n_1 + \dots + n_R = N}} \! \frac{N!}{n_1! n_2! \dotsm n_R!} \prod_{i=1}^R (z_i A_t^{(i)})^{n_i} \right) \nonumber\\
&= \sum_{N=0}^\infty \sum_{\substack{n_1, \dotsc, n_R \geq 0 \\[2pt] n_1 + \dots + n_R = N}} \left( \prod_{l=1}^R \frac{(-z_l)^{n_l}}{n_l!} \right) \EE_{(x,j)}\left( \prod_{i=1}^R (A_t^{(i)})^{n_i} \right). 
\label{eq_expandmgf}
\end{align}
The generalized Kac's moment formula~\eqref{eq_Kac_multi} with the functions $V_k(x) = \idop_{L_k}(x)$ (and formulated for the spider) becomes
\begin{equation}\label{KAC1}
\EE_{(x,j)}\left(\prod_{i=1}^R (A_t^{(i)})^{n_i}\right) = \sum_{i=1}^R n_i \int_0^\infty \beta_i m(\diff y) \int_0^t p(s;(x,j),(y,i)) \EE_{(y,i)}\left(\frac{\prod_{k=1}^R (A_{t-s}^{(k)})^{n_k}}{A_{t-s}^{(i)}}\right) \diff s, 
\end{equation}
where $p(s;(x,j),(y,i))$ denotes the transition density of the spider. 
Next, equation~\eqref{KAC1} is inserted in \eqref{eq_expandmgf}, bearing in mind that by Remark~\ref{rem_kaczeros} this can be done even when some of the values $n_1,\dotsc,n_R$ are zero. 
The exception is the term for $N=0$, since at least one of the $n_i$ has to be strictly positive, so this term (which evaluates to 1) is separated from the rest of the sum. 
This yields
\begin{align*}
&\EE_{(x,j)}\left(\mathrm{exp}\bigg(- \sum_{i=1}^R z_i A_t^{(i)}\bigg)\right) \\
&\; = 1 + \sum_{N=1}^\infty \sum_{\substack{n_1, \dotsc, n_R \geq 0 \\[2pt] n_1 + \dots + n_R = N}} \left( \prod_{l=1}^R \frac{(-z_l)^{n_l}}{n_l!} \right) \sum_{i=1}^R n_i \int_0^\infty \beta_i\,m(\diff y) \int_0^t p(s;(x,j),(y,i)) \EE_{(y,i)}\left(\frac{\prod_{k=1}^R (A_{t-s}^{(k)})^{n_k}}{A_{t-s}^{(i)}}\right) \diff s \\
&\; = 1 + \sum_{i=1}^R \int_0^\infty\beta_i\, m(\diff y) \int_0^t p(s;(x,j),(y,i)) \sum_{N=0}^\infty \sum_{\substack{n_1, \dotsc, n_R \geq 0 \\[2pt] n_1 + \dots + n_R = N}} (-z_i) \left( \prod_{l=1}^R \frac{(-z_l)^{n_l}}{n_l!} \right) \EE_{(y,i)}\left(\prod_{k=1}^R (A_{t-s}^{(k)})^{n_k}\right) \diff s \\
&\; = 1 - \sum_{i=1}^R z_i \int_0^\infty\beta_i\, m(\diff y) \int_0^t p(s;(x,j),(y,i)) \EE_{(y,i)}\left(\mathrm{exp}\bigg(- \sum_{i=1}^R z_i A_{t-s}^{(i)}\bigg)\right) \diff s. 
\end{align*}
Note that the summation indices $N$ and $n_i$ have both been shifted by~1 in the third step and that in the last step~\eqref{eq_expandmgf} has been applied again. 
From~\eqref{eq_mgfTasLaplace} and the above, we obtain
\begin{align}
&\EE_{(x,j)}\left(\mathrm{exp}\bigg(- \sum_{i=1}^R z_i A_T^{(i)}\bigg) \right) 
= \lambda \, \LL_t\left\{ \EE_{(x,j)}\left(\mathrm{exp}\bigg(- \sum_{i=1}^R z_i A_t^{(i)}\bigg) \right) \right\}(\lambda) \nonumber\\
&= \lambda \cdot \frac{1}{\lambda} - \lambda \sum_{i=1}^R z_i \int_0^\infty \LL_t\left\{ \int_0^t p(s;(x,j),(y,i)) \EE_{(y,i)}\left(\mathrm{exp}\bigg(- \sum_{i=1}^R z_i A_{t-s}^{(i)}\bigg)\right) \diff s \right\}(\lambda)\,\beta_i\, m(\diff y) \nonumber\\
&= 1 - \lambda \sum_{i=1}^R z_i \int_0^\infty \mathbf{g}_\lambda((x,j),(y,i)) \LL_t\left\{\EE_{(y,i)}\left(\mathrm{exp}\bigg(- \sum_{i=1}^R z_i A_{t}^{(i)}\bigg)\right) \right\}(\lambda) \,\beta_i\,m(\diff y) \nonumber\\
&= 1 - \sum_{i=1}^R z_i \int_0^\infty \mathbf{g}_\lambda((x,j),(y,i)) \EE_{(y,i)}\left(\mathrm{exp}\bigg(- \sum_{i=1}^R z_i A_{T}^{(i)}\bigg)\right)\beta_i\, m(\diff y). \label{eq_momgeninty}
\end{align} 
We now insert~\eqref{eq_momgendepx} into the right hand side of~\eqref{eq_momgeninty} and let~$x\to 0$ on both sides. This gives
\begin{align*} 
\EE_{\orig}\left(\mathrm{exp}\bigg(- \sum_{i=1}^R z_i A_T^{(i)}\bigg) \right) 
&= 1 - \sum_{i=1}^R z_i \int_0^\infty \mathbf{g}_\lambda(\orig,(y,i)) \frac{\lambda}{\lambda + z_i} \left(1- \EE_{(y,i)}\bigg(\ee^{-(\lambda + z_i) H_\orig} \bigg) \right)\beta_i\, m(\diff y) \\ 
& - \EE_\orig\left(\mathrm{exp}\bigg(- \sum_{i=1}^R z_i A_T^{(i)}\bigg)\right) \sum_{i=1}^R z_i \int_0^\infty \mathbf{g}_\lambda(\orig,(y,i)) \EE_{(y,i)}\bigg(\ee^{-(\lambda + z_i) H_\orig} \bigg)\beta_i\, m(\diff y), 
\end{align*} 
which, when solved for the left hand side expression, results in the claimed formula~\eqref{eq_momgenspider}. 
\end{proof}

In the next result, we connect our formula more transparently with the result given by Yano~\cite{Yano2017} in Theorem~3.5; see also Theorem~4 in  Barlow, Pitman and Yor~\cite{BarlowPitmanYor1989arcsinus}, where where the formula is presented for Bessel spiders (see Section~\ref{sec_bessel}). 
For a Bessel spider, $c_\lambda$ is as given in \eqref{eq_stable}, and this means that the inverse of the local time at 0 of the underlying reflecting Bessel process is a stable subordinator. 

\begin{corollary}
\label{YA}
For  $r\in\{1,2,\dotsc, R\}$ and $z_i>0,\, i=1,2,\dotsc,r,$ 
\begin{equation} \label{eq_momgenyano_1}
\EE_\orig\left(\mathrm{exp}\bigg(- \sum_{i=1}^r z_i A_T^{(i)}\bigg)\right) 
= \frac{\displaystyle 1 -\sum_{j=1}^r \beta_j +\sum_{j=1}^r  \frac{\lambda\,\beta_j}{\lambda + z_j} \frac{c_{\lambda+z_j}}{c_\lambda}}
{\displaystyle 1 -\sum_{j=1}^r \beta_j +\sum_{j=1}^r \beta_j \frac{c_{\lambda+z_j}}{c_\lambda}},
\end{equation}
where (cf. \eqref{cr1}) 
\begin{equation*}
c_\lambda:=-\frac{d}{dS}\varphi_\lambda(0+)>0.
\end{equation*}
In particular, for $z_i>0$, $i=1,2,\dotsc,R$,
\begin{equation} \label{eq_momgenyano_0}
\EE_\orig\left(\mathrm{exp}\bigg(- \sum_{i=1}^R z_i A_T^{(i)}\bigg)\right) 
= \frac{\displaystyle \sum_{j=1}^R   \frac{\lambda\,\beta_j\,c_{\lambda+z_j}}{\lambda + z_j}}
{\displaystyle \sum_{j=1}^R \beta_j \, c_{\lambda+z_j}}
\end{equation}
\end{corollary}

\begin{proof}
Recall from \eqref{CY} that
\[
\mathbf{g}_\lambda(\orig,(y,j)) = \frac 1{c_\lambda}\,{\varphi_\lambda(y)}
\]
and from \eqref{CT} (with the normalization $\varphi(0)=1$ as in \eqref{norm}) that
\[
\EE_{(y,j)}(\ee^{-(\lambda+z_j) H_\orig}) = \varphi_{\lambda + z_j}(y). 
\]
Hence, \eqref{eq_momgenspider} can be rewritten
\[
\EE_\orig\left(\mathrm{exp}\bigg(-\sum_{i=1}^r z_i A_T^{(i)}\bigg)\right) =
\frac{\displaystyle 1-\sum_{j=1}^r \frac{\lambda\, \beta_j\,z_j}{c_\lambda\,(\lambda + z_j)} \left( \int_0^\infty \varphi_\lambda(y) \,m(\diff y) - \int_0^\infty \varphi_\lambda(y)\varphi_{\lambda+z_j}(y) \,m(\diff y) \right)}{\displaystyle 1+ \sum_{j=1}^r \frac{ \beta_j\,z_j}{c_\lambda} \int_0^\infty \varphi_\lambda(y)\varphi_{\lambda+z_j}(y) \,m(\diff y)}. 
\] 
For the integrals, we have (cf.~\cite{SalminenStenlund2021}, proof of Corollary~1)
\begin{align*}
\lambda \int_0^\infty \varphi_\lambda(y)\,m(\diff y) &=c_\lambda, \\
z_j \int_0^\infty \varphi_\lambda(y)\varphi_{\lambda+z_j}(y) \,m(\diff y) &= -c_\lambda+c_{\lambda+z_j}
\end{align*}
by which the previous equation becomes
\begin{align*}
\EE_\orig\left(\mathrm{exp}\bigg(-\sum_{i=1}^r z_i A_T^{(i)}\bigg)\right) &= \frac{\displaystyle 1 - \sum_{j=1}^r \left(\frac{\lambda\,\beta_j\, z_j}{\lambda + z_j}\right) \left( \frac{1}{\lambda} + \frac{1}{z_j} \right) +\sum_{j=1}^r  \frac{\lambda\,\beta_j}{\lambda + z_j} \frac{c_{\lambda+z_j}}{c_\lambda}}
{\displaystyle 1 -\sum_{j=1}^r \beta_j +\sum_{j=1}^r \beta_j \frac{c_{\lambda+z_j}}{c_\lambda}}
\end{align*}
from which  \eqref{eq_momgenyano_1} easily follows, and  \eqref{eq_momgenyano_0} is immediate since
$
\sum_{j=1}^R \beta_j =1.
$
\end{proof}

\section{Examples}
In this section we highlight our results by analyzing a few different diffusion spiders, first and foremost Bessel spiders. 
For the Brownian spider, which is an important special case of Bessel spiders, it is possible to pursue the formulas further, and this evaluation is presented in a subsection of its own. 
Finally, we make some comments concerning occupation times for Walsh Brownian motion.

\subsection{Bessel spider}
\label{sec_bessel}

We now apply the result in Theorem~\ref{thm_momrec} on a particular example of a diffusion spider. 
Namely, let $\X$ be a Bessel spider with $R\geq 2$ legs. 
We define this as a diffusion spider that behaves like a Bessel process with parameter $\nu\in(-1,0)$ (i.e., dimension $2+2\nu$) on each leg and has the corresponding excursion probabilities $\beta_i>0$, $i=1,2,\dotsc R$, such that $\beta_1+\dots+\beta_R=1$. 

The Bessel spider has the self-similar property, which means that the recurrence equation in~\eqref{eq_selfrec} applies. 
Recall that this recurrence hinges on the factors $D_k^{(i)}$ given in~\eqref{eq_Dk}.
For the purpose of finding $\mathbf{g}_\lambda(\orig,(y,i))$, that is, where the point $y$ is on a particular leg $L_i$, we follow the procedure leading to Theorem~\ref{Thrm:GreenKernel1}. 
For the reflected Bessel diffusion on $[0,+\infty)$, we have from  
\cite{BorodinSalminen2015} (p.~137) that 
\[
m(dx)= 2x^{2v+1} dx,\qquad S(x)= -\frac 1{2\nu} x^{-2\nu}, \qquad  \varphi_\lambda(x)= x^{-\nu}K_\nu(x\sqrt{2\lambda}),
\]
where $K_\nu$ is a modified Bessel function of the second kind. Then
\begin{equation}
\label{eq_stable}
c_\lambda:=-\frac{d}{dS}\varphi(0+)=  \biggl(\frac 2{\sqrt{2\lambda}} \biggr)^\nu \Gamma(\nu + 1) 
= 2^{\nu/2}\Gamma(\nu + 1)\,\lambda^{-\nu/2}
\end{equation}
and, hence, 
\[
\mathbf{g}_\lambda(\orig,(y,i)) = \frac{1}{ \Gamma(\nu + 1)} \biggl( \frac{\sqrt{2\lambda}}{2} \biggr)^\nu y^{-\nu} K_\nu (y\sqrt{2\lambda}). 
\]
Note that $\mathbf{g}_\lambda(\orig,(y,i))$ is the same for all~$i$. 
Similarly, the hitting time $H_\orig$ when starting in $y\in L_i$ corresponds precisely to the hitting time of zero in the reflected Bessel process. 
The values of $D_k$ are calculated as in the proof of Theorem~3 in \cite{SalminenStenlund2021} (note, however, that in ibid.\ a different normalization is used for $m$ and~$S$), and we have for any $\lambda>0$ that
\begin{equation}\label{eq_DkBessel}
D_k^{(i)}(\lambda) = - \beta_i \binom{\nu + k - 1}{k} = - \frac{\beta_i}{k!} \sum_{j=1}^k \stirlingone{k}{j} \nu^j,
\end{equation}
where $\stirlingone{n}{k}$ are unsigned Stirling numbers of the first kind.
In the rest of this section, we will drop $\lambda$ and only write $D_k^{(i)}$, as its value does not depend on $\lambda$. 

As explained in Section~\ref{sec_singleleg}, when only considering the occupation time on a single leg $L_i$, we can directly use the earlier obtained results for skew two-sided Bessel processes. 
Hence, by \cite{SalminenStenlund2021} (Theorem~4), the $n$th moment of the occupation time on $L_i$ up to time~1 is given by 
\begin{equation}\label{eq_besselmom}
\EE_\orig\left((A^{(i)}_1)^n\right) = \sum_{l=1}^{n}\sum_{k=1}^{l} (-1)^{k-1} \frac{\Gamma(k)}{\Gamma(n)} \stirlingone{n}{l} \stirlingtwo{l}{k} \nu^{l-1} \beta_i^{k}, 
\end{equation}
where $\stirlingtwo{n}{k}$~are Stirling numbers of the second kind. 

Using the recurrence equation in Theorem~\ref{thm_momrec}, the result in~\eqref{eq_besselmom} is here extended to an explicit formula for the joint moments of the occupation times on multiple legs in a Bessel spider with $R\geq 2$ legs. 
With the numbering of legs being arbitrary, it should be clear that the formula -- although written for the first $r$ legs of the spider -- holds when considering the occupation times on any number $r$ of the $R$ legs. 
Contrary to the recursive formula in Theorem~\ref{thm_momrec} (see Remark~\ref{rem_r1}), this formula also holds when $r=1$. 

\begin{theorem}\label{thm_besrec}
For any $r\in\{1,\dotsc,R\}$ and $n_1,\dotsc,n_r \geq1$, 
\begin{equation} \label{eq_besseljointmom}
\EE_\orig\left(\prod_{i=1}^r (A_1^{(i)})^{n_i}\right) = \sum_{1\leq k_1 \leq l_1 \leq n_1} \!\!\dotsm\!\! \sum_{1\leq k_r \leq l_r \leq n_r} (-1)^{K -1} \frac{\Gamma(K)}{\Gamma(N)} \nu^{L-1} \prod_{j=1}^r \stirlingone{n_j}{l_j} \stirlingtwo{l_j}{k_j} \beta_j^{k_j}, 
\end{equation}
where $N=n_1+\dots+n_r$, $K=k_1+\dots+k_r$ and $L=l_1+\dots+l_r$. 
\end{theorem}

A proof of the theorem is given in Appendix~\ref{app_A}. 
For a particularly simple instance of the theorem above, consider the joint \emph{first} moment of the occupation times on $r$~legs in the Bessel spider. 

\begin{corollary}
For any $r\in\{1,\dotsc,R\}$, 
\begin{equation} \label{eq_besselfirstmom}
\EE_\orig\left(A_1^{(1)} A_1^{(2)} \dotsm A_1^{(r)}\right) = (-\nu)^{r-1} \beta_1 \beta_2 \dotsm \beta_r.  
\end{equation}
\end{corollary}

\begin{proof}
Immediate from~\eqref{eq_besseljointmom} with $n_1=\dots=n_r=1$. 
\end{proof}

\subsection{Brownian spider}
\label{BSP}
The special case of a Bessel spider with the parameter $\nu = -\frac{1}{2}$ is the \emph{Brownian spider} mentioned in the introduction, also known as Walsh Brownian motion on a finite number of legs. 
In this case, the result in Theorem~\ref{thm_besrec} has the following, somewhat simpler expression. 

\begin{theorem}\label{thm_brownrec}
Let  $\X$ be a Brownian spider and let $A_1^{(i)}$ be the occupation time on leg~$L_i$ up to time~1. 
For any $r\in\{1,\dotsc,R\}$ and $n_1,\dotsc,n_r \geq1$, 
\begin{equation} \label{eq_brownianjointmom}
\EE_\orig\left(\prod_{i=1}^r (A_1^{(i)})^{n_i}\right) = \sum_{k_1=1}^{n_1} \dotsm \sum_{k_r=1}^{n_r} 2^{-(2N-K-1)} \frac{\Gamma(K)}{\Gamma(N)} \prod_{j=1}^r \frac{\Gamma(2n_j-k_j)\,\beta_j^{k_j}}{\Gamma(k_j)\Gamma(n_j-k_j+1)} , 
\end{equation}
where $N=n_1+\dots+n_r$ and $K=k_1+\dots+k_r$. 
\end{theorem}

\begin{proof}
The result follows from~\eqref{eq_besseljointmom} and the identity
\begin{equation*}
\sum_{i=k}^n \stirlingone{n}{i} \stirlingtwo{i}{k} (-2)^{n-i} = (-1)^{n-k} \frac{(2n-k-1)!}{2^{n-k} (k-1)! (n-k)!} =: b(n,k), 
\end{equation*}
where $b(n,k)$ is a (signed) Bessel number of the first kind. 
For proofs of this identity and some related ones, see~\cite{Stenlund2022} and~\cite{YangQiao2011}. 
\end{proof}

\subsection{Walsh Brownian motion}
\label{WBSP}

Finally, we briefly return to the Walsh Brownian motion in its original form. 
As the state space can be the entire $\RR^2$ and is not restricted to a spider graph with a fixed number of legs, in this section we follow Walsh's terminology and talk about ``rays'' rather than ``legs'', although it should be clear that the meaning is the same. 
Here the diffusion behaves like a Brownian motion on each ray, and when it reaches the origin, the direction~$\theta$ of the next ray is selected according to some given distribution on $[0,2\pi)$. 
As we have already studied the case when this distribution is discrete -- i.e., the number of rays is at most countable -- we now consider a continuous distribution. 

The diffusion will, almost surely, choose a new direction every time it reaches $\orig$, so that it visits no ray more than once. 
Furthermore, the probability of visiting a particular ray (i.e., a ray whose angle is a fixed value) is zero. 
For this reason, it is not meaningful to consider the occupation times on specific rays in this case. 
Rather, we can consider the occupation times within sectors of the $\RR^2$ plane. 
Let $0=\theta_0<\theta_1<\theta_2<\dots<\theta_R=2\pi$ be fixed angles and let $S_i$ consist of all points with angle in $[\theta_{i-1}, \theta_i)$, so that $\RR^2$ is partitioned into $R$ non-overlapping sectors $S_1, S_2, \dotsc, S_R$. 
If $\Theta(X_t)$ denotes the angle of the ray on which the diffusion $X$ is located at time~$t$, then
\[
A_t^{(S_i)} = \int_0^t \idop_{[\theta_{i-1}, \theta_i)}(\Theta(X_s)) \diff s
\]
is the occupation time of the diffusion within sector~$S_i$ up to time~$t$. 

With respect to the occupation time on a sector, the outcome is the same as if all rays within the sector were combined and mapped onto a single ray. 
Therefore, the occupation times of a Walsh Brownian motion on the sectors $S_1, \dotsc, S_R$ precisely correspond to the occupation times on the $R$ legs of a Brownian spider. 
Thus, the result in Theorem~\ref{thm_brownrec} applies for the occupation times on sectors of a Walsh Brownian motion, with $\beta_i$ being equal to the probability of selecting an angle within sector~$S_i$ when at the origin. 
Naturally, if the diffusion behaves like a Bessel process with parameter $\nu\in(-1,0)$ on each ray (this could, perhaps, be called a ``Walsh Bessel process''), then Theorem~\ref{thm_besrec} applies instead.

\clearpage
\begin{appendices}

\bigskip
\section{}
\label{app_A}
\begin{proof}[Proof of Theorem~\ref{thm_besrec}]
The known equation~\eqref{eq_besselmom} coincides with \eqref{eq_besseljointmom} where $r=1$, showing that the theorem holds in that particular case. 
As in the proof of Theorem~\ref{thm_momrec}, we here prove the statement for $r=2$, say the two legs $L_1$ and $L_2$ in the Bessel spider. 
This will be enough to demonstrate the procedure, which can then readily be repeated for a larger value of~$r$ with more tedious but hardly more difficult work. 

To begin with, we repeat the statement in~\eqref{eq_besseljointmom} for $r=2$: 
\begin{equation}
\EE_\orig\left((A_1^{(1)})^{n_1} (A_1^{(2)})^{n_2} \right) 
= \sum_{l_1=1}^{n_1} \sum_{k_1=1}^{l_1} \sum_{l_2=1}^{n_2} \sum_{k_2=1}^{l_2} (-1)^{k_1+k_2 -1} \frac{\Gamma(k_1+k_2)}{\Gamma(n_1+n_2)} \stirlingone{n_1}{l_1} \stirlingtwo{l_1}{k_1}\stirlingone{n_2}{l_2} \stirlingtwo{l_2}{k_2} \nu^{l_1+l_2-1} \beta_1^{k_1} \beta_2^{k_2}. \label{eq_2jointmom}
\end{equation}
We prove this statement by induction using the recurrence in~\eqref{eq_selfrec} and the known moment formula~\eqref{eq_besselmom} for the occupation time on a single leg. 
For the simplest case $n_1=n_2=1$, we see from~\eqref{eq_selfrec} that
\[
\EE_\orig \left( A^{(1)}_{1} A^{(2)}_{1} \right) = \frac{1}{2} D_1^{(1)} \EE_\orig \left( A^{(2)}_{1} \right) + \frac{1}{2} D_1^{(2)} \EE_\orig \left( A^{(1)}_{1} \right) = -\beta_1 \beta_2 \nu, 
\]
since $D_1^{(i)} = - \beta_i \nu$ and $\EE_\orig(A^{(i)}_{1}) = \beta_i$. Thus, \eqref{eq_2jointmom} holds in this case. 
Assume now that~\eqref{eq_2jointmom} holds whenever $\{1\leq n_1\leq a-1, 1\leq n_2\leq b\}$ or $\{1\leq n_1\leq a, 1\leq n_2\leq b-1\}$ for some integers $a,b\geq 1$. 
We proceed to show that then~\eqref{eq_2jointmom} holds also for $n_1=a, n_2=b$. 

First, we apply the recurrence equation~\eqref{eq_selfrec} to get 
\begin{align}
\EE_\orig \left( (A^{(1)}_{1})^{a} (A^{(2)}_{1})^{b} \right) &= \frac{D_a^{(1)}}{\binom{a+b}{a}} \EE_\orig \left( (A^{(2)}_{1})^{b} \right) + \sum_{i=1}^{a-1} \frac{\binom{a}{i}}{\binom{a+b}{i}} D_i^{(1)} \EE_\orig \left( (A^{(1)}_{1})^{a-i} (A^{(2)}_{1})^{b} \right) \nonumber\\
& + \frac{D_b^{(2)}}{\binom{a+b}{b}} \EE_\orig \left( (A^{(1)}_{1})^{a} \right) + \sum_{i=1}^{b-1} \frac{\binom{b}{i}}{\binom{a+b}{i}} D_i^{(2)} \EE_\orig \left( (A^{(1)}_{1})^{a} (A^{(2)}_{1})^{b-i} \right). \label{eq_4terms}
\end{align}
Here we have separated the terms with moments of the occupation time on only one leg, for which~\eqref{eq_besselmom} applies, and the terms with joint moments of the occupation times on both legs, for which we can apply the induction assumption. 
In the first case, we insert the expressions in~\eqref{eq_DkBessel} and~\eqref{eq_besselmom} to get
\begin{align}
 \frac{D_a^{(1)}}{\binom{a+b}{a}} \EE_\orig \left( (A^{(2)}_{1})^{b} \right) 
&= \frac{1}{\binom{a+b}{a}} \left( - \frac{\beta_1}{a!} \sum_{l_1=1}^a \stirlingone{a}{l_1} \nu^{l_1} \right) \sum_{l_2=1}^{b}\sum_{k_2=1}^{l_2} (-1)^{k_2-1} \frac{\Gamma(k_2)}{\Gamma(b)} \stirlingone{b}{l_2} \stirlingtwo{l_2}{k_2} \nu^{l_2-1} \beta_2^{k_2} \nonumber \\
&= b \sum_{l_1=1}^a \sum_{k_1=1}^1 \sum_{l_2=1}^{b}\sum_{k_2=1}^{l_2} (-1)^{k_1+k_2-1} \frac{\Gamma(k_1+k_2-1)}{\Gamma(a+b+1)} \stirlingone{a}{l_1} \stirlingtwo{l_1}{k_1} \stirlingone{b}{l_2} \stirlingtwo{l_2}{k_2} \nu^{l_1+l_2-1} \beta_1^{k_1} \beta_2^{k_2}. \label{eq_DaEA2b}
\end{align}
Note that the variable $k_1$ only takes the value 1 here, but it is nevertheless added so that the expression above resembles the form of~\eqref{eq_2jointmom} more closely. 
Next, we turn to the following term in~\eqref{eq_4terms}, assuming for the moment that $a>1$ so that the sum is not empty. Inserting~\eqref{eq_DkBessel} and the induction assumption, we obtain
\begin{align}
& \sum_{i=1}^{a-1} \frac{\binom{a}{i}}{\binom{a+b}{i}} D_i^{(1)} \EE_\orig \left( (A^{(1)}_{1})^{a-i} (A^{(2)}_{1})^{b} \right) 
= \sum_{i=1}^{a-1} \frac{\binom{a}{i}}{\binom{a+b}{b+i}} D_{a-i}^{(1)} \EE_\orig \left( (A^{(1)}_{1})^{i} (A^{(2)}_{1})^{b} \right) \nonumber\\
& = \sum_{i=1}^{a-1} \frac{\binom{a}{i}}{\binom{a+b}{b+i}} \left( - \frac{\beta_1}{(a-i)!} \sum_{j=1}^{a-i} \stirlingone{a-i}{j} \nu^j \right) \sum_{\substack{1\leq k_1 \leq l_1 \leq i, \\[2pt] 1\leq k_2 \leq l_2 \leq b}} \hspace{-0.7em} (-1)^{k_1+k_2 -1} \frac{\Gamma(k_1+k_2)}{\Gamma(b+i)} \stirlingone{i}{l_1} \stirlingone{b}{l_2} \stirlingtwo{l_1}{k_1} \stirlingtwo{l_2}{k_2} \nu^{l_1+l_2-1} \beta_1^{k_1} \beta_2^{k_2} \nonumber\\
& = \hspace{-0.4em} \sum_{\substack{1\leq k_1 \leq l_1 \leq i \leq j \leq a-1, \\[2pt] 1\leq k_2 \leq l_2 \leq b}} \hspace{-2em} (-1)^{k_1+k_2}  \nu^{a+l_1+l_2-j-1} \beta_1^{k_1+1} \beta_2^{k_2} \frac{(b+i) \Gamma(k_1+k_2)}{\Gamma(a+b+1)} \binom{a}{i} \stirlingone{i}{l_1} \stirlingone{b}{l_2} \stirlingtwo{l_1}{k_1} \stirlingtwo{l_2}{k_2} \stirlingone{a-i}{a-j} \nonumber\\
& = \hspace{-0.6em} \sum_{\substack{1\leq k_1 \leq j \leq a-1, \\[2pt] 1\leq k_2 \leq l_2 \leq b}} \hspace{-1.1em} (-1)^{k_1+k_2}  \nu^{l_2+j} \beta_1^{k_1+1} \beta_2^{k_2} \frac{\Gamma(k_1+k_2)}{\Gamma(a+b+1)} \stirlingone{b}{l_2} \stirlingtwo{l_2}{k_2} \left( \sum_{l_1=k_1}^j \stirlingtwo{l_1}{k_1} \sum_{i=l_1}^{a+l_1-j-1} \hspace{-0.7em} (b+i) \binom{a}{i}\stirlingone{i}{l_1} \stirlingone{a-i}{j-l_1+1} \right)\!, \label{eq_sextuple}
\end{align}
where, in the last step, we have changed the order of summation according to the pattern
\begin{align*}
\sum_{k=1}^{n} \sum_{l=k}^{n} \sum_{i=l}^{n} \sum_{j=i}^{n} f(i, j, k, l) 
&= \sum_{k=1}^{n} \sum_{l=k}^{n} \sum_{i=l}^{n} \sum_{j=l}^{n+l-i} f(i, n+l-j, k, l) \\
&= \sum_{k=1}^{n} \sum_{j=k}^{n} \sum_{l=k}^{j} \sum_{i=l}^{n+l-j} f(i, n+l-j, k, l).
\end{align*}
To simplify the expression further, we use the two closely related identities
\[
\sum_{i=k}^{n-m} \stirlingone{i}{k} \stirlingone{n-i}{m} \binom{n}{i} = \binom{k+m}{k} \stirlingone{n}{k+m}, \qquad
\sum_{i=k}^{n-m} \stirlingone{i+1}{k+1} \stirlingone{n-i}{m} \binom{n}{i} = \binom{k+m}{k} \stirlingone{n+1}{k+m+1}, 
\]
the first of which is well known~\cite{GrahamKnuthPatashnik1994} and is also utilized in the proof of the latter (see Lemma~2 in~\cite{SalminenStenlund2021}). 
Applying these identities, we get that the innermost sum in~\eqref{eq_sextuple} is equal to
\begin{align*}
&\sum_{i=l_1}^{a+l_1-j-1} (b+i) \binom{a}{i}\stirlingone{i}{l_1} \stirlingone{a-i}{j-l_1+1} \\
&\qquad = b \sum_{i=l_1}^{a-(j-l_1+1)} \binom{a}{i}\stirlingone{i}{l_1} \stirlingone{a-i}{j-l_1+1} + a \sum_{i=l_1}^{a-(j-l_1+1)} \binom{a-1}{i-1}\stirlingone{i}{l_1} \stirlingone{a-i}{j-l_1+1} \\
&\qquad = b \binom{j+1}{l_1} \stirlingone{a}{j+1} + a \binom{j}{l_1-1} \stirlingone{a}{j+1} \\
&\qquad = \stirlingone{a}{j+1} \left( b \binom{j}{l_1} + (a+b) \binom{j}{l_1-1} \right). 
\end{align*}
The expression inside the parenthesis in~\eqref{eq_sextuple} becomes
\begin{align*}
\sum_{l_1=k_1}^j \stirlingtwo{l_1}{k_1} \sum_{i=l_1}^{a+l_1-j-1} \!(b+i) \binom{a}{i}\stirlingone{i}{l_1} \stirlingone{a-i}{j-l_1+1} 
& = \stirlingone{a}{j+1} \left( b \sum_{l_1=k_1}^j \stirlingtwo{l_1}{k_1} \binom{j}{l_1} + (a+b) \sum_{l_1=k_1}^j \stirlingtwo{l_1}{k_1} \binom{j}{l_1-1} \!\right) \\
& = \stirlingone{a}{j+1}\stirlingtwo{j+1}{k_1+1} \bigl( b+ (a+b)k_1 \bigr), 
\end{align*}
where the second step follows by the identities
\[
\sum_{i=k}^n \stirlingtwo{i}{k} \binom{n}{i} = \stirlingtwo{n+1}{k+1}, \qquad 
\sum_{i=k}^n \stirlingtwo{i}{k} \binom{n}{i-1} = k \stirlingtwo{n+1}{k+1}, 
\]
see Equation~(6.15) in~\cite{GrahamKnuthPatashnik1994} and the proof of Theorem~4 in~\cite{SalminenStenlund2021}, respectively. 
Inserting this into~\eqref{eq_sextuple} gives
\begin{align*}
& \sum_{i=1}^{a-1} \frac{\binom{a}{i}}{\binom{a+b}{i}} D_i^{(1)} \EE_\orig \left( (A^{(1)}_{1})^{a-i} (A^{(2)}_{1})^{b} \right) \\
&\qquad = \!\sum_{\substack{1\leq k_1 \leq j \leq a-1, \\[2pt] 1\leq k_2 \leq l_2 \leq b}} \hspace{-12pt} (-1)^{k_1+k_2}  \nu^{l_2+j} \beta_1^{k_1+1} \beta_2^{k_2} \frac{\Gamma(k_1+k_2)}{\Gamma(a+b+1)} \stirlingone{b}{l_2} \stirlingtwo{l_2}{k_2} \stirlingone{a}{j+1}\stirlingtwo{j+1}{k_1+1} \bigl( b+ (a+b)k_1 \bigr) \\
&\qquad = \!\sum_{\substack{2\leq k_1 \leq j \leq a, \\[2pt] 1\leq k_2 \leq l_2 \leq b}} \hspace{-12pt} (-1)^{k_1+k_2-1}  \nu^{l_2+j-1} \beta_1^{k_1} \beta_2^{k_2} \frac{\Gamma(k_1+k_2-1)}{\Gamma(a+b+1)} \stirlingone{a}{j} \stirlingtwo{j}{k_1} \stirlingone{b}{l_2} \stirlingtwo{l_2}{k_2} \bigl( b+ (a+b)(k_1-1) \bigr). 
\end{align*}
When this is combined with~\eqref{eq_DaEA2b}, replacing the summation index $j$ with $l_1$ in the process, the result is 
\begin{align*}
& \frac{D_a^{(1)}}{\binom{a+b}{a}} \EE_\orig \left( (A^{(2)}_{1})^{b} \right) + \sum_{i=1}^{a-1} \frac{\binom{a}{i}}{\binom{a+b}{i}} D_i^{(1)} \EE_\orig \left( (A^{(1)}_{1})^{a-i} (A^{(2)}_{1})^{b} \right) \\
&\qquad = \sum_{\substack{1\leq k_1 \leq l_1 \leq a, \\[2pt] 1\leq k_2 \leq l_2 \leq b}} \hspace{-12pt} (-1)^{k_1+k_2-1}  \frac{\Gamma(k_1+k_2-1)}{\Gamma(a+b+1)} \stirlingone{a}{l_1} \stirlingtwo{l_1}{k_1} \stirlingone{b}{l_2} \stirlingtwo{l_2}{k_2} \nu^{l_1+l_2-1} \beta_1^{k_1} \beta_2^{k_2} \bigl( b + (a+b)(k_1-1) \bigr). 
\end{align*}
Note that this coincides with~\eqref{eq_DaEA2b} when $a=1$, so the temporary assumption $a>1$ is not necessary for this expression to hold. 
This is half the right hand side of~\eqref{eq_4terms}. 
We obtain the other half simply by interchanging the roles of $a$ and $b$ in the expression above (renaming the summation indices accordingly). 
Thus, adding all the terms together, we get
\begin{align*}
\EE_\orig \left( (A^{(1)}_{1})^{a} (A^{(2)}_{1})^{b} \right) 
&= \sum_{\substack{1\leq k_1 \leq l_1 \leq a, \\[2pt] 1\leq k_2 \leq l_2 \leq b}} \hspace{-12pt} (-1)^{k_1+k_2-1}  \frac{\Gamma(k_1+k_2-1)}{\Gamma(a+b+1)} \stirlingone{a}{l_1} \stirlingtwo{l_1}{k_1} \stirlingone{b}{l_2} \stirlingtwo{l_2}{k_2} \nu^{l_1+l_2-1} \beta_1^{k_1} \beta_2^{k_2} \cdot \\[-12pt]
& \hspace{6cm} \Bigl( b + (a+b)(k_1-1) + a + (a+b)(k_2-1) \Bigr) \\[8pt]
&= \sum_{\substack{1\leq k_1 \leq l_1 \leq a, \\[2pt] 1\leq k_2 \leq l_2 \leq b}} \hspace{-12pt} (-1)^{k_1+k_2-1}  \frac{\Gamma(k_1+k_2)}{\Gamma(a+b)} \stirlingone{a}{l_1} \stirlingtwo{l_1}{k_1} \stirlingone{b}{l_2} \stirlingtwo{l_2}{k_2} \nu^{l_1+l_2-1} \beta_1^{k_1} \beta_2^{k_2}, 
\end{align*}
which proves that~\eqref{eq_2jointmom} indeed holds for $n_1=a, n_2=b$. By induction, it holds for any $n_1, n_2 \geq 1$, and the theorem is thereby proved when $r=2$. 
As pointed out in the beginning of the proof, the same method can be repeated for increasingly larger values of $r$ as well. 
However, due to the long expressions we include only the proof given above and trust that the reader can recognize how it generalizes to~\eqref{eq_besseljointmom} for higher~$r$. 
\end{proof}

\end{appendices}

\clearpage

\end{document}